\documentclass[12pt]{article}

\usepackage{amsmath, amssymb, amscd, amsthm}

\usepackage[all]{xy}

\theoremstyle{plain}
\newtheorem{thm}{Theorem}

\newtheorem{theorem}{Theorem}
\newtheorem{lemma}[theorem]{Lemma}

\theoremstyle{definition}

\newtheorem{remark}[theorem]{Remark}

\newtheorem{example}[theorem]{Example}

\newcommand\bA{{\mathbb A}}

\newcommand\bG{{\mathbb G}}

\newcommand\bP{{\mathbb P}}

\newcommand\bfHom{\mathbf{Hom}}
\newcommand\bfR{\mathbf{R}}

\newcommand\cG{{\mathcal G}}

\newcommand\cL{{\mathcal L}}
\newcommand\cM{{\mathcal M}}

\newcommand\cO{{\mathcal O}}

\newcommand\ant{{\rm ant}}

\newcommand\ev{{\rm ev}}

\newcommand\id{{\rm id}}

\newcommand\pr{{\rm pr}}

\newcommand\red{{\rm red}}
\newcommand\reg{{\rm reg}}

\DeclareMathOperator\Gal{Gal}
\newcommand\GL{{\rm GL}}

\DeclareMathOperator\Hilb{Hilb}
\DeclareMathOperator\Hom{Hom}

\DeclareMathOperator\R{R}

\newcommand\Spec{{\rm Spec}}

\DeclareMathOperator\Supp{Supp}

\numberwithin{equation}{section}

\title{On models of algebraic group actions}

\author{Michel Brion}

\date{}

\begin{document}

\maketitle

\begin{abstract}
We show that every action of a smooth algebraic 
group on a variety admits a normal projective model. 
Along the way, we present new proofs of some basic 
results on algebraic transformation groups, 
including Weil's regularization theorem.
\end{abstract}

\section{Introduction}
\label{sec:int}

The main motivation for this paper comes from
Weil's regularization theorem (see \cite{Weil}):
for any rational action of a smooth connected algebraic 
group $G$ on a smooth variety $X$, there exists a $G$-action 
on a smooth variety $Y$ which is $G$-equivariantly birationally 
isomorphic to $X$. This was extended to non-connected 
groups by Rosenlicht (see \cite[Thm.~1]{Rosenlicht}), 
with modern proofs over an algebraically closed
field by Zaitsev and Kraft (see \cite{Zaitsev,Kraft}).

One would like to have a model $Y$ satisfying
additional geometric properties. When working over 
an arbitrary field $k$ as in \cite{Weil,Rosenlicht}, 
there may exist no smooth projective model
(see Remark \ref{rem:imperfect} for details).
Is there a regular projective one? This question has
an affirmative answer when $X$ is a curve: then $G$ acts 
on its regular projective model, see e.g.
\cite[Prop.~5]{Demazure}. For surfaces, an affirmative 
answer is known in various degrees of generality, 
as a key ingredient in the classification of algebraic 
subgroups of groups of birational transformations
(see \cite[\S 2]{Blanc}, \cite[Prop.~2.4]{Fong}, 
\cite[Prop.~2.3]{SZ}).

In this paper, we show that the model $Y$ can be taken 
projective and normal. If $k$ has characteristic $0$, 
it follows that $Y$ can be taken projective and 
smooth, by using equivariant resolution of singularities 
(see \cite[Prop.~3.9.1, Thm.~3.36]{Kollar}). 
Similarly, if $X$ is a surface over an arbitrary field
$k$, then $Y$ can be taken projective and regular, 
by iterating the processes of blowing up the singular 
locus and normalizing (see 
\cite[Rem.~B, p.~155]{Lipman}). In higher 
dimensions and positive characteristics, the existence 
of an equivariant desingularization is an open problem.

When $G$ is connected, the existence of 
a normal projective model is deduced in 
\cite[Cor.~3]{Br17a} from two general results:

\begin{enumerate}

\item[(1)] 
Every normal $G$-variety admits a covering 
by $G$-stable quasi-projective open subsets.

\item[(2)] 
Every normal quasi-projective $G$-variety
admits an equivariant completion by a normal 
projective $G$-variety.

\end{enumerate}

\noindent
Note that (1) fails for non-connected groups,
already for the group with two elements
in view of an example of Hironaka (see
\cite{Hironaka} and \cite[Chap.~IV, \S 3]{MFK}). 
But it extends in a weaker form:

\begin{thm}\label{thm:gen}
Let $G$ be a smooth algebraic group, and $X$ 
a $G$-variety. Then $X$ contains a $G$-stable 
dense open subset which is regular and 
quasi-projective.
\end{thm}

Also, (2) extends without change:

\begin{thm}\label{thm:eq}
Let $G$ be a smooth algebraic group, and $X$ a normal 
quasi-projective $G$-variety. Then $X$ admits 
an equivariant completion by a normal projective 
$G$-variety.
\end{thm}

If in addition $G$ is affine, then every normal $G$-variety 
(not necessarily quasi-projective)
admits an equivariant completion by a normal proper
$G$-variety, in view of a result of Sumihiro 
(see \cite[Thm.~3]{Su74},  \cite[Thm.~4.13]{Su75}). 
We do not know whether this still holds for 
an arbitrary smooth algebraic group $G$.

This paper is organized as follows. Section 
\ref{sec:ptgen} begins with a new proof
of (1) for a connected group $G$; then we derive 
Theorem \ref{thm:gen} as an easy consequence.
Section \ref{sec:torsors} contains auxiliary results 
on torsors. In Section \ref{sec:afb}, we first recall
the classical construction of the associated fiber bundle, 
also known as contracted product; in loose terms,
it constructs a $G$-variety from an $H$-variety, 
where $H$ is a subgroup of $G$. Then we obtain
an existence result for such bundles (Theorem 
\ref{thm:afb}) which is not used in the proofs of
the main theorems, but which should have independent 
interest.  Section \ref{sec:wr} begins with another 
construction of a $G$-variety from an $H$-variety, 
which is somehow dual to the associated fiber bundle 
construction, and based on Weil restriction. 
This yields the main ingredient for the proof 
of Theorem \ref{thm:eq}. 
In Section \ref{sec:wrt}, we present a new proof 
of Weil's regularization theorem,
which is shorter but less self-contained than
the original one.
Actions of non-smooth groups are considered
in Section \ref{sec:ns}; for these, we show that
projective models still exist, but normal 
projective models do not. We also extend
Weil's regularization theorem to this setting
(Theorem \ref{thm:weil}).

\medskip

\noindent
{\bf Notation and conventions.}
We consider schemes over a field $k$ with 
algebraic closure $\bar{k}$, separable closure 
$k_s$ and absolute Galois group $\cG = \Gal(k_s/k)$.
Morphisms and products of schemes are understood 
to be over $k$ as well. Schemes are assumed to be 
separated and of finite type unless explicitly 
mentioned. 
A \emph{variety} $X$ is an integral scheme. 
The function field of $X$ is denoted by $k(X)$.

An \emph{algebraic group} $G$ is a group scheme;
we denote by $e \in G(k)$ its neutral element.
By a \emph{subgroup} $H$ of $G$, we mean a  
subgroup scheme; then $H$ is closed in $G$. 
A $G$-\emph{scheme} is a scheme $X$ 
equipped with a $G$-action
\[ a : G \times X \longrightarrow X, 
\quad (g,x) \longmapsto g \cdot x. \]
A morphism of $G$-schemes $f : X \to Y$ is
\emph{equivariant} if 
$f(g \cdot x) = g \cdot f(x)$
identically on $G \times X$.

Given a field extension $K/k$ and a $k$-scheme $X$, 
we denote by $X_K$ the $K$-scheme 
$X \times_{\Spec(k)} \Spec(K)$.

\section{Proof of Theorem \ref{thm:gen}}
\label{sec:ptgen}

\noindent
{\sc Proof of (1)}.
Recall the statement: \emph{let $G$ be a smooth 
connected algebraic group, and $X$ a normal 
$G$-variety. Then $X$ admits a covering by 
$G$-stable quasi-projective open subsets.} 

This follows from a result of Benoist asserting
that every normal variety contains only finitely 
many maximal quasi-projective open subsets (see 
\cite[Thm.~9]{Benoist}). An alternative proof,
using methods from Raynaud's work on the 
quasi-projectivity of homogeneous spaces 
(see \cite{Raynaud}), 
is presented in \cite[Sec.~3]{Br17a}. We will 
recall the very beginning of this proof, and then use
a more direct argument adapted from the proof of 
the quasi-projectivity of algebraic groups 
in \cite[Lem.~39.8.7]{SP}. This bypasses the most
technical developments of \cite{Br17a}.

Let $U$ be an open subset of $X$. Since the action
$a : G \times X \to X$ is a smooth morphism
(see e.g. \cite[Lem.~1.68]{Milne}),
$G \cdot U = a(G \times U)$ is open in $X$; 
also, $G \cdot U$ contains $U$ and is clearly $G$-stable. 
Thus, it suffices to show that $G \cdot U$ is 
quasi-projective for any affine $U$.

For this, we may assume that $X = G \cdot U$.
Since $U$ is affine, its complement is the support
of an effective Weil divisor $D$ on $X$.
We now show that $D$ is a Cartier divisor. 
The product $G \times X$ is 
a normal variety, since $G$ is a smooth, 
geometrically integral variety and $X$ is 
a normal variety (use \cite[IV.6.8.5]{EGA}). 
Also, $a$ is smooth and the Weil divisor 
$G \times D = \pr_2^*(D)$ on $G \times X$ 
contains no fiber of $a$ in its support, since 
$X = G \cdot U$. By the Ramanujam-Samuel theorem
(see \cite[IV.21.14.3]{EGA}),
it follows that $G \times D$ is a Cartier divisor.
Thus, $D = (e,  \id_X)^*(G \times D)$ 
is a Cartier divisor as well.

It is shown in \cite[Prop.~3.3]{Br17a} that 
$D$ is ample, by using a version of the
theorem of the square. We will rather construct
from $D$ an ample divisor $E$ on $X$, via a norm
argument taken from \cite[Tag 0BF7]{SP}.

Since $G$ is smooth, there exists a dense open 
subset $V$ of $G$ and an \'etale morphism
$f : V \to \bA^n$, where $n = \dim(G)$ 
(see \cite[IV.17.11.4]{EGA}).
Replacing $V$ with a smaller open subset, we
may assume that $f$ is finite over its image $W$,
a dense open subset of $\bA^n$. In particular, 
$f : V \to W$ is finite and locally free of 
constant degree $d$.

Denote by $\varphi : V \times X \to X$ 
the pull-back of $a$. Then 
$\cO_{V \times X}(\varphi^*(D))$ 
is an invertible sheaf on $V \times X$, equipped
with a section with divisor of zeroes 
$\varphi^*(D)$. Thus, the norm of 
$\cO_{V \times X}(\varphi^*(D))$ with respect
to the finite locally free morphism
$f \times \id_X : V \times X \to W \times X$
is an invertible sheaf $\cM$ on $W \times X$,
equipped with a section whose divisor of zeroes
$E$ satisfies  
$\Supp(E) = (f \times \id_X)(\Supp (\varphi^*(D)))$
(see \cite[II.6.5]{EGA} for the construction and 
properties of the norm).

Since $W$ is open in $\bA^n$, there exists a
invertible sheaf $\cL$ on $X$ such that
$\cM = \pr_X^*(\cL)$; moreover, $\cL \simeq \cM_w$ 
for any $w \in W(k)$, where $\cM_w$ denotes the
pull-back of $\cM$ under the morphism 
$(w,\id_X): X \to W \times X$, $x \mapsto (w,x)$. 
We claim that $\cL$ is ample.

To prove this claim, we may assume that $k$ is
algebraically closed, since $\cL$ is ample
if and only if $\cL_{\bar{k}}$ is ample on $X_{\bar{k}}$
(see \cite[VIII.5.8]{SGA1}). For any $w \in W(k)$, 
we denote by $E_w$ the pull-back of $E$ under 
$(w,\id_X)$. Since $f$ is finite 
and \'etale of degree $d$, we have 
$f^{-1}(w) = \{ g_1, \ldots, g_d \}$
where $g_i = g_i(w) \in V(k)$ for $i = 1, \ldots, d$.
Then  
$\Supp(E_w) = \bigcup_{i=1}^d g_i^{-1} \Supp(D)$
and hence
$X \setminus \Supp(E_w) = \bigcap_{i= 1}^d g_i^{-1} U$ 
is affine. Also, note that $\Supp(E_w)$ is the 
zero set of a section of $\cM_w \simeq \cL$. In view of 
\cite[II.4.5.2]{EGA}, it suffices to show that 
$X$ is the union of the $X \setminus \Supp(E_w)$, 
where $w \in W(k)$. 
Since $X$ is of finite type,
it suffices in turn to check the equality
$X(k) = \bigcup_{w \in W(k)} 
X(k) \setminus \Supp(E_w)$.

Let $x \in X(k)$ and denote by $V_x$ the set of
$g \in V$ such that $x \in g^{-1} U$.
Then $V_x$ is a dense open subset of $V$ as 
$X = G \cdot U$. Thus, there exists $w \in W(k)$ 
such that $f^{-1}(w)$ is contained in 
$V_x$. Then $x \in X \setminus \Supp(E_w)$ 
as desired.
\qed

\medskip

\noindent
{\sc Proof of Theorem \ref{thm:gen}.}
The regular locus $X_{\reg}$ is a dense open 
subset of $X$. We show that $X_{\reg}$
is $G$-stable. Since $G$ is smooth, 
$G \times X_{\reg}$ is regular by \cite[IV.6.8.5]{EGA}.
As the action morphism $a : G \times X \to X$ 
is flat, it sends $G \times X_{\reg}$ to $X_{\reg}$ 
in view of \cite[IV.6.5.2]{EGA}. 

Thus, we may assume that $X$ is regular. In view of (1), 
there exists a dense open quasi-projective subset 
$U \subset X$, stable by the neutral component $G^0$. 
Thus, $U_{k_s}$ is a dense open quasi-projective subset 
of $X_{k_s}$, stable by $G^0_{k_s}$. Moreover, 
$X_{k_s}$ is regular by \cite[IV.6.7.6, IV.6.8.5]{EGA}. 
Also, there are only finitely
many translates $g \cdot U$, where $g \in G(k_s)$
(since $G(k_s)/G^0(k_s)$ is finite). So 
$V = \bigcap_{g \in G(k_s)} g \cdot U_{k_s}$ 
is a dense open subset of $U_{k_s}$, stable by
$G(k_s)$. As $G$ is smooth, $G(k_s)$ is dense
in $G_{k_s}$ and hence $V$ is stable by
$G_{k_s}$; it is also stable by the Galois
group $\cG$, and quasi-projective.

The assertion follows from this by a standard 
argument of Galois descent. More specifically,
since $X$ is of finite type, there exists
a finite Galois extension $K/k$ and an open
subset $W$ of $X_K$ such that $V = W_{k_s}$. 
Then $W$ is quasi-projective 
(see e.g. \cite[Prop.~14.55]{GW}). Also, 
$W$ is stable by $G_K$ and by the Galois group
$\cG_K$ of $K/k$. So the quotient
$W/\cG_K \subset X_K/\cG_K = X$
satisfies the assertion.
\qed

\section{Torsors}
\label{sec:torsors}

Let $G$ be an algebraic group. 
A $G$-\emph{torsor} is a faithfully flat morphism 
of schemes $f : X \to Y$, where $X$ is equipped 
with a $G$-action $a$ such that $f$ is $G$-invariant and 
$G \times X 
\stackrel{\sim}{\longrightarrow}
X \times_Y X$ via $(a,\pr_2)$.
This isomorphism identifies the two projections
$p_1,p_2 : X \times_Y X \to X$ with 
$a,\pr_2: G \times X \to X$.
As a consequence, for any $G$-scheme $Z$, 
we have an isomorphism
\[ f^* : \Hom(Y,Z) \stackrel{\sim}{\longrightarrow}
\Hom^G(X,Z), \quad 
\varphi \longmapsto \varphi \circ f, \]
where the right-hand side denotes the set of 
$G$-invariant morphisms. 

We now record three auxiliary results:

\begin{lemma}\label{lem:tnormal}
Let $G$ be an algebraic group, and
$f : X \to Y$ a $G$-torsor, where $X$ is
a normal variety. Then $Y$ is a normal variety. 
\end{lemma}

\begin{proof}
Since $X$ is normal and $f$ is faithfully flat, 
$Y$ is normal by \cite[IV.6.5.4]{EGA}.
Also, $Y$ is irreducible, as $X$ is irreducible 
and $f$ is surjective. 
\end{proof}

\begin{lemma}\label{lem:taction}
Let 
$1 \longrightarrow N \longrightarrow G
\stackrel{\pi}{\longrightarrow} Q 
\longrightarrow 1$
be an exact sequence of algebraic groups.
Let $X$ be a $G$-scheme, and $f : X \to Y$
an $N$-torsor. 

\begin{enumerate}

\item[{\rm (i)}]
There is a unique action of $Q$ on $Y$ 
such that the diagram
\[ \xymatrix{
G \times X \ar[r]^-{a} 
\ar[d]_{\pi \times f} 
& X \ar[d]^{f} \\
Q \times Y \ar[r]^-{b} & Y 
}\]
commutes.

\item[{\rm (ii)}]
For any $Q$-scheme $Z$
viewed as a $G$-scheme via $\pi$, 
the composition with $f$ yields an isomorphism
$\Hom^Q(Y,Z) 
\stackrel{\sim}{\longrightarrow}
\Hom^G(X,Z)$.

\end{enumerate}

\end{lemma}

\begin{proof}
(i) As $N$ is a normal subgroup of $G$,
and $f$ is $N$-invariant,
the composition 
\[ G \times X \stackrel{a}{\longrightarrow}
X \stackrel{f}{\longrightarrow} Y \]
is invariant under the action of $N$ on
$G \times X$ via $n \cdot (g,x) = (g,n \cdot x)$.
Since the morphism
$\id_G \times f: G \times X \to G \times Y$
is an $N$-torsor, it follows that there exists
a unique morphism 
$a' : G \times Y \to Y$ 
such that the diagram
\[ \xymatrix{
G \times X \ar[r]^-{a} 
\ar[d]_{\id_G \times f} 
& X \ar[d]^{f} \\
G \times Y \ar[r]^-{a'} & Y 
}\]
commutes. So this diagram is cartesian (e.g.,
since its vertical arrows are $N$-torsors).
Using the uniqueness property for 
descent of morphisms, one may check that $a'$ 
is a $G$-action. Moreover, $N$ acts trivially
on $Y$ via $a$, and hence via $a'$ by 
uniqueness again. Thus, $a'$ is invariant
under the $N$-action on $G \times Y$ via
$n \cdot (g,y) = (g n^{-1},y)$. Since
$\pi \times \id_Y : G \times Y \to Q \times Y$
is an $N$-torsor, it follows that $a'$ factors
uniquely through a morphism 
$b : Q \times Y \to Y$. One may then check
as above that $b$ is a $Q$-action.

(ii) This follows from the isomorphism
\[ f^ * : \Hom(Y,Z) \stackrel{\sim}{\longrightarrow}
\Hom^N(X,Z) \]
by taking invariants of $G$, or equivalently
of $Q$.
\end{proof}

We also record a variant of \cite[Prop.~4.2.4]{Br17b}:

\begin{lemma}\label{lem:tabelian}
Let $f : X \to Y$ be a torsor under an abelian 
variety $A$, where $X$ is smooth. 
Then there exists an $A$-equivariant morphism
$\psi : X \to A/A[n]$ for some positive 
integer $n$, where $A[n]$ denotes the schematic 
kernel of the multiplication by $n$ in $A$.
\end{lemma}

\begin{proof}
We first reduce to the case where $X$ is
a variety. The irreducible components 
$X_1,\ldots,X_r$ of $X$ are disjoint and
stable by $A$. If there exist 
$A$-equivariant morphisms 
$\psi_i : X_i \to A/A[n_i]$ for $i = 1,\ldots,r$,
then we obtain the desired morphism
$\psi$ by setting $n = \prod_{i=1}^r n_i$ and 
$\psi \vert_{X_i} = \frac{n}{n_i} \psi_i$
for all $i$.

Therefore, we may assume that $X$ is a variety;
then so is $Y$. The base change
$X \times_Y \Spec \, k(Y)$
is an $A_{k(Y)}$-torsor over $k(Y)$. 
By \cite[Lem.~8.22]{Milne}, there exists 
an $A_{k(Y)}$-equivariant morphism 
\[ \varphi : X \times_Y \Spec \, k(Y) 
\longrightarrow 
A_{k(Y)}/A_{k(Y)}[n] = (A/A[n])_{k(Y)} \]
for some positive integer $n$. Composing
$\varphi$ with the natural morphisms
\[ \Spec \, k(X) \longrightarrow 
X \times_Y \Spec \, k(Y),
 \quad 
(A/A[n])_{k(Y)} \longrightarrow A/A[n] \] 
yields a rational $A$-equivariant map
$\psi : X \dasharrow A/A[n]$. Since
$X$ is smooth, $\psi$ is a morphism
by Weil's extension theorem
(see \cite[Thm.~4.4/1]{BLR}).
 \end{proof}

\section{Associated fiber bundles}
\label{sec:afb}

Consider an algebraic group $G$, a subgroup 
$H$, and the corresponding homogeneous space 
$G/H$ with quotient morphism $q: G \to G/H$ and base
point $o = q(e)$. Then $q$ is an $H$-torsor 
(see e.g. \cite[5.c)]{Milne}).
Given an $H$-scheme $Y$, the 
\emph{associated fiber bundle} is a scheme
$Z$ which fits in a cartesian square
\[ \xymatrix{
G \times Y \ar[r]^-{\pr_1} \ar[d]_{\varphi} 
& G \ar[d]^q \\
Z \ar[r]^-{\psi} & G/H,
}\]
where $\varphi$ is invariant under the action
of $H$ on $G \times Y$ via 
\[ h \cdot (g,y) = (g h^{-1}, h \cdot y).  \]

If such a scheme $Z$ exists, then $\varphi$ 
is an $H$-torsor (the pull-back of the $H$-torsor
$q$ under $\psi$) and $\psi$ is faithfully flat 
(since so are $\varphi$, $q$ and $\pr_1$). 
Also, $G$ acts on $G \times Y$ via left 
multiplication on itself; 
this action commutes with the $H$-action, and 
$\pr_1$ is equivariant. By Lemma \ref{lem:taction}, 
this yields a $G$-action on $Z$ such that $\psi$ 
is equivariant. In particular, the fiber
$Z_o$ is an $H$-scheme, equivariantly
isomorphic to $Y$. We denote $Z$ by $G \times^H Y$.

Conversely, if $Z$ is a $G$-scheme and 
$\psi : Z \to G/H$ a $G$-equivariant morphism 
such that $Z_o \simeq Y$ as $H$-schemes, 
then one may easily check that $Z \simeq G \times^H Y$ 
(see \cite[\S 2.5]{Br17a}).

The associated fiber bundle need not exist, 
as shown again by Hironaka's example (see
\cite{BB} for details).  
But for any $G$-scheme $Y$, 
the associated fiber bundle $G \times^H Y$ 
exists and is isomorphic to $G/H \times Y$, 
where $G$ acts diagonally; this identifies 
$\psi$ with the first projection. 

Another instance where $G \times^H Y$ exists
is when $Y$ is $H$-\emph{quasi-projective}, 
i.e., it admits an ample $H$-linearized
line bundle; then $G\times ^H Y$ is
quasi-projective (this follows from 
\cite[Prop.~7.1]{MFK} or from 
\cite[10.2/6]{BLR}).

We now obtain a further existence result:

\begin{thm}\label{thm:afb}
Let $G$ be a smooth algebraic group, 
$H$ a subgroup, and $Y$ a normal quasi-projective
$H$-variety. Then the associated fiber bundle
$G \times^H Y$ exists and is a normal 
quasi-projective $G$-variety.
\end{thm}

\begin{proof}
By \cite[Thm.~2]{Br17b}, there is a smallest
normal subgroup $N$ of $H$ such that the quotient
$Q = H/N$ is proper; moreover, $N$ is affine and 
connected. Since $Y$ is normal and quasi-projective,
it is $N$-quasi-projective in view of 
\cite[Lem.~2.9]{Br17a}. Thus, $G \times^N Y$ exists 
and is quasi-projective. Also, $G \times Y$ is normal 
(as $G$ is smooth and $Y$ is normal), and hence 
$G \times^N Y$ is normal by Lemma \ref{lem:tnormal}.

The group $G \times H$ acts on $G \times Y$
via $(g,h) \cdot (g',y) = (gg' h^{-1}, h \cdot y)$.
In view of Lemma \ref{lem:taction}, this yields 
an action of $Q$ on $G \times^N Y$ which commutes 
with the $G$-action and such that the canonical
morphism
\[ \psi : G \times^N Y \longrightarrow G/N\]
is equivariant.

We have a unique exact sequence of algebraic groups
\[ 1 \longrightarrow A \longrightarrow Q
 \longrightarrow F \longrightarrow 1, \]
 where $A$ is an abelian variety and $F$ is finite
 (see \cite[Lem.~3.3.7]{Br17b}). Denote by $I$
 the pull-back of $A$ in $H$. Then we obtain two
 exact sequences
 \[ 1 \longrightarrow N \longrightarrow I
 \longrightarrow A \longrightarrow 1, \quad
 1 \longrightarrow I \longrightarrow H
 \longrightarrow F \longrightarrow 1. \]

We now show that the associated fiber bundle
$G \times^I Y$ exists and is quasi-projective.
For this, we consider the natural morphism
$G/N \to G/I$ which is an $A$-torsor
(see e.g. \cite[Prop.~2.8.4]{Br17b}); moreover,
$G/N$ is smooth. 
By Lemma \ref{lem:tabelian},
it follows that there exists an $A$-equivariant
morphism $\varphi : G/N \to A/A[n]$ for some 
positive integer $n$. Composing $\varphi$ with
the $A$-equivariant morphism $\psi$ yields 
an $A$-equivariant morphism 
$G \times^N Y \to A/A[n]$, and hence an 
$A$-equivariant isomorphism
\[ G \times^N Y \simeq A \times^{A[n]} Z, \]
where $Z$ is a closed $A[n]$-stable subscheme
of $G \times^N Y$. In particular, $Z$ is
quasi-projective and its $A[n]$-action is free.

As $A[n]$ is finite, it follows that there is 
an $A[n]$-torsor $f: Z \to W$, where $W$ is 
a scheme (see \cite[III.2.2.3, III.2.6.1]{DG}).
We claim that $W$ is quasi-projective.
Indeed, $Z$ admits an ample $A[n]$-linearized
line bundle $L$ (as follows e.g. from
\cite[Prop.~2.12]{Br17a}). By descent,
we have $L \simeq f^*(M)$ for a line bundle
$M$ on $W$, unique up to isomorphism.
Since $f$ is finite and locally free, 
$M$ is ample in view of \cite[II.6.6.1]{EGA};
this proves the claim.

Next, the projection $\pr_2 : A \times Z \to Z$
lies in a commutative square
\[ 
\xymatrix{ A \times Z \ar[r]^-{\pr_2} \ar[d] 
& Z \ar[d]^f \\
A \times^{A[n]} Z \ar[r]^-{\pi} & W
}\]
where the vertical arrows are $A[n]$-torsors.
Thus, this square is cartesian and the morphism
$\pi : G \times^N Y = A \times^{A[n]} Z \to W$ 
is an $A$-torsor. By Lemma \ref{lem:taction},
it follows that $W$ is equipped with an action 
of $G \times Q/A = G \times F$. 
Since the composition
\[ G \times^N Y \longrightarrow G/N 
\longrightarrow G/I \] 
is $G$-equivariant and $A$-invariant, this yields
a commutative diagram 
\[ 
\xymatrix{ G\times^N Y \ar[r] \ar[d] 
& W \ar[d] \\
G/N \ar[r] & G/I
}\]
where all arrows are $G$-equivariant, and 
the horizontal arrows are $A$-torsors.
As a consequence, this diagram is cartesian,
so that $W = G \times^I Y$; also, 
the $F$-action on $W$ is free. Arguing as 
for the construction of $W$,
this yields a cartesian square 
\[ 
\xymatrix{ W \ar[r] \ar[d] 
& V \ar[d] \\
G/I \ar[r] & G/H
}\]
where the horizontal arrows are $F$-torsors
and $V$ is a $G$-scheme. Thus, 
$V$ satisfies the assertions.
\end{proof}

\section{Weil restriction and proof of 
Theorem \ref{thm:eq}}
\label{sec:wr}

Given a morphism of schemes $\psi: Z \to W$,
recall that the Weil restriction functor 
$\bfR_{W/k}(Z/W) = \bfR_W(Z)$ is the contravariant
functor from $k$-schemes (not necessarily 
of finite type) to sets, which sends every
such scheme $S$ to the set of sections of 
the morphism 
$\psi \times \id_S : Z \times S \to W \times S$.
These sections may be identified with the 
morphisms $\sigma : W \times S \to Z$
such that $\psi \circ \sigma = \pr_1$,
and this identifies $\bfR_W(Z)$ with the
functor considered in \cite[7.6]{BLR}.

Next, assume that $Z$ is quasi-projective,
and $W$ is projective. Then in view of
\cite[p.~267]{Grothendieck}, the functor 
$\bfR_W(Z)$ is represented by an open subscheme 
$\R_W(Z)$ of the Hilbert scheme $\Hilb(Z)$.
This is obtained by sending every section to 
its schematic image; recall that the $S$-points
of $\Hilb(Z)$ are the closed subschemes $V$ 
of $Z \times S$ which are flat over $S$, 
and the schematic images of sections
are those $V$ such that 
$V \stackrel{\sim}{\longrightarrow} W \times S$
via $\psi \times \id_S$. 

Since $\Hilb(Z)$ is a disjoint union 
of open and closed subschemes of finite type
(for which the Hilbert polynomial is specified),
$\R_W(Z)$ satisfies the same property. 
Thus, $\Hilb(Z)$ and $\R_W(Z)$ 
are locally of finite type. 

As mentioned in \cite[\S 7.6]{BLR}, Weil restriction 
is compatible with base change. In particular,
for any $w \in W(k)$ with schematic fiber $Z_w$, 
we have a morphism of functors 
$\bfR_W(Z) \to Z_w$
(where $Z_w$ is identified with its functor of points),
and hence a morphism of schemes
\[ \ev_w : \R_W(Z) \longrightarrow Z_w, \quad 
\sigma \longmapsto \sigma(w,-) \]
under the above assumptions on $Z$ and $W$. 

Also, if $Z$ and $W$ are equipped with actions 
of an algebraic group $G$ such that 
$\psi$ is equivariant, then
we obtain compatible $G$-actions on $\R_W(Z)$
and $\Hilb(Z)$. If in addition $w$ is fixed by
a subgroup $H$ of $G$, then $\ev_w$ 
is $H$-equivariant.

We now consider an algebraic group $G$, 
a subgroup $H$ of $G$ and an $H$-scheme $Y$. 
We assume that the associated fiber bundle 
\[ \psi : Z = G \times^H Y \longrightarrow G/H = W \]
exists, and $Z$ is quasi-projective
(by Theorem \ref{thm:afb}, 
this holds if $G$ is smooth and $Y$ is normal and 
quasi-projective).
We assume in addition that the homogeneous
space $G/H$ is proper, or equivalently projective.
Then the preceding discussion yields:

\begin{lemma}\label{lem:rep}
With the above assumptions, the functor
$\bfR_W(Z)$ is represented by an open $G$-stable
subscheme $\R_W(Z)$ of $\Hilb(Z)$. Moreover,
the evaluation morphism
$\ev_o : \R_W(Z) \to Y$ is $H$-equivariant.
\end{lemma}

Next, let $X$ be a $G$-scheme, and 
$f : X \longrightarrow Y$
an $H$-equivariant morphism. Then
\[ \id_G \times f : G \times X 
\longrightarrow G \times Y \]
is a $G \times H$-equivariant morphism of schemes
over $G$, where $G \times H$ acts 
on $G \times X$ and $G \times Y$ via 
$(g,h) \cdot (g',z) = (g g' h^{-1}, h \cdot z)$.
By Lemma \ref{lem:taction}, this defines 
a $G$-equivariant morphism 
\[ G \times^H f : G \times^H X 
\longrightarrow G \times^H Y = Z \]
of schemes over $W$, which pulls back to
$f : X \to Y$ on fibers at $o$. Also, since 
$X$ is a $G$-scheme, the morphism
\[ G \times X \longrightarrow W \times X,
\quad 
(g,x) \longmapsto (g \cdot o, g \cdot x) \]
factors through a $G$-equivariant isomorphism 
$G \times^H X 
\stackrel{\sim}{\longrightarrow}
W \times X$
of schemes over $W$, which pulls back to
$o \times \id_X$ on fibers at $o$. This defines a 
$G$-equivariant morphism $W \times X \to Z$
of schemes over $W$, which pulls back to
$f$ on fibers at $o$. In turn, this yields
a $G$-equivariant morphism
\[ F : X \longrightarrow \R_W(Z) \]
such that $\ev_o \circ F = f$.

\begin{lemma}\label{lem:imm}
With the above notation, assume that $f$ is 
an immersion. Then $F$ is an immersion as well.
\end{lemma}

\begin{proof}
For any $x \in X$, the composition of
homomorphisms of local rings
\[ \cO_{Y,f(x)} \longrightarrow \cO_{\R_W(Z),F(x)}
 \longrightarrow \cO_{X,x} \]
is an isomorphism, and hence the homomorphism
$\cO_{\R_W(Z),F(x)} \to \cO_{X,x}$
is surjective. By \cite[I.3.4.4]{DG}, it follows
that $F$ is a local immersion at $x$. Also,
$F$ is injective, since so is $f$. But
every injective local immersion is an immersion,
since $X$ is irreducible (see \cite[I.3.4.5]{DG}).
\end{proof}

\begin{remark}\label{rem:fun}
We sketch an alternative approach to the
above Weil restriction functor $\bfR_W(Z)$. 
Consider the functor $\bfHom^H(G,Y)$ sending 
every scheme $S$ to the set of $H$-equivariant 
morphisms $f : G \times S \to Y$, where $H$ acts 
on $G \times S$ via right multiplication on $G$.
Then $G$ acts on $\bfHom^H(G,Y)$ via left
multiplication on itself. Moreover, evaluating
at $e$ yields a morphism of functors 
\[ \varepsilon : 
\bfHom^H(G,Y) \longrightarrow Y,
\quad f \longmapsto f(e,-), \]
which is $H$-equivariant (here $Y$ is identified
with its functor of points). With this notation,
$\bfHom^H(G,Y)$ is $G$-equivariantly isomorphic 
to $\bfR_W(Z)$, and this isomorphism identifies 
$\varepsilon$ with $\ev_e$. 

This follows from the classical
identification of morphisms with their graph:
more specifically, for any scheme $S$, the data
of a morphism $f : G \times S \to Y$ 
is equivalent to that of its graph, a closed 
subscheme $\Gamma_f$ of $G \times S \times Y$
such that the projection 
$\pr_{12} : G \times S \times Y 
\to G \times S$
is an isomorphism. Moreover, $f$ is 
$H$-equivariant if and only if $\Gamma_f$
is stable under the $H$-action on 
$G \times S \times Y$ via 
$h\cdot (g,s,y) = (gh^{-1}, s, h \cdot y)$.
By descent for the $H$-torsor
$\varphi_S : G \times S \times Y \to Z \times S$,
the closed subschemes of $G \times S \times Y$
obtained in this way are exactly the 
pull-backs under $\varphi_S$ of the closed 
subschemes $\Gamma \subset Z \times S$
such that 
$\Gamma \stackrel{\sim}{\longrightarrow}
W \times S$ 
via $\psi \times \id_S$. This yields the
desired isomorphism of functors. Its $G$-equivariance
and compatibility with $\varepsilon$ and $\ev_e$
are easily verified.

For any $G$-scheme $X$, one may show that the map 
\[ \Hom^G(X, \bfHom^H(G,Y)) \longrightarrow 
\Hom^H(X,Y), 
\quad f \longmapsto \varepsilon \circ f \]
is an isomorphism, by using the canonical
isomorphism
\[ \Hom(X,\bfHom(G,Y)) 
\stackrel{\sim}{\longrightarrow} 
\Hom(G \times X,Y) \]
(see \cite[I.1.7.1]{SGA3}) and taking 
$G \times H$-invariants. 
Equivalently,  $\bfHom^H(G,Y)$
is right adjoint to the the restriction functor
from $G$-schemes to $H$-schemes. A left adjoint
to this restriction functor is provided by 
the associated fiber bundle: denoting 
by $i : Y \to G \times^H Y$
the canonical ($H$-equivariant) closed immersion, 
the map
\[ \Hom^G(G \times^H Y, X) \longrightarrow 
\Hom^H(Y,X), \quad f \longmapsto f \circ i \]
is an isomorphism as well.
\end{remark}

\noindent
{\sc Proof of Theorem \ref{thm:eq}.}
We first argue as in the beginning of the proof of
Theorem \ref{thm:afb}.
By \cite[Thm.~2]{Br17b}, $G$ has a smallest normal 
subgroup $H$ such that $G/H$ is proper; moreover, 
$H$ is affine and connected. Since $X$ is normal
and quasi-projective, it is $H$-quasi-projective
in view of \cite[Lem.~2.9]{Br17a}. 
Thus, there exists an $H$-equivariant open
immersion $f : X \to Y$, where $Y$ is an
$H$-projective scheme; then 
the associated fiber bundle $Z = G \times^H Y$ 
exists and is quasi-projective. Since $Y$
is proper, the projection 
$\pr_1 : G \times Y \to G$ is proper as well.
By descent (see \cite[IV.2.7.1]{EGA}), 
it follows that the morphism 
$\psi : Z \to G/H = W$ is proper as well.
As $W$ is proper, we see that $Z$ is proper,
and hence projective. 
By Lemma \ref{lem:imm}, the $H$-equivariant
immersion of $X$ in $Y$ yields a $G$-equivariant
immersion of $X$ in $\R_W(Z)$, and hence in $\Hilb(Z)$.
Since $X$ is a variety, its schematic image
in $\Hilb(Z)$ is a projective $G$-variety. Taking
its normalization yields the desired normal projective
equivariant completion.
\qed

\begin{remark}\label{rem:imperfect}
Assume that $k$ is imperfect. Then there exist (many)
$k$-forms $G$ of the additive group $\bG_a$
such that the regular completion $X$ of $G$
is not smooth (see \cite{Russell}; here $G$
is viewed as a smooth affine curve). The 
$G$-action on itself by translation extends
uniquely to a $G$-action on $X$, which is
the unique normal projective model of $G$.
In particular, $G$ admits no smooth
projective model.
\end{remark}

\section{Proof of Weil's regularization theorem}
\label{sec:wrt}

Throughout this section,  
we consider a smooth algebraic group $G$ 
and a smooth variety $X$. 
We first recall some notions and results 
from \cite[\S 3]{Demazure}.   

A \emph{rational action} of $G$ on $X$ 
is a rational map 
\[ a  : G \times X \dasharrow X, 
\quad (g,x) \longmapsto g \cdot x \]
which satisfies the following two properties:

\begin{enumerate}

\item[{\rm (i)}] The rational map
\[ (\pr_1,a) : G \times X \dasharrow G \times X,
\quad (g,x) \longmapsto (g, g \cdot x) \]
is dominant.

\item[{\rm (ii)}] The rational maps
\[ a \circ (\id_G \times a), 
a \circ (m \times \id_X) :
G \times G \times X \dasharrow X, \quad
(g,h,x) \longmapsto 
g \cdot (h \cdot x), gh \cdot x
\]
are equal, 
where $m : G \times G \to G$
denotes the multiplication.

\end{enumerate}

We denote by $V$ the domain of definition
of $a$; this is an open dense subset of 
$G \times X$. Since $G$ and $X$ are smooth,
$V$ is smooth as well. For any scheme $S$ and any
$g \in G(S)$,  $x \in X(S)$, we say that
$g \cdot x$ \emph{is defined} if $(g,x) \in V(S)$.
Then (ii) is equivalent to

\begin{enumerate}

\item[{\rm (iii)}] For any scheme $S$ and 
any $g,h \in G(S)$,  $x \in X(S)$, 
if $h \cdot x$ and $g \cdot (h \cdot x)$
are defined, 
then $gh \cdot x$ is defined and
equals $g \cdot (h \cdot x)$.

\end{enumerate}

For any $g \in G$ with residue field $\kappa(g)$,  
the fiber $\pr_1^{-1}(g) \subset G \times X$
is identified with $X_{\kappa(g)}$ and this identifies
$V \cap \pr_1^{-1}(g)$ with an open subset
$V_g$ of $X_{\kappa(g)}$. 
In view of \cite[Lem.~1]{Demazure}, 
$V_g$ is dense in $X_{\kappa(g)}$ and
the morphism
\[ a_g : V_g \longrightarrow X_{\kappa(g)},
\quad x \longmapsto g \cdot x \]
is dominant. In particular, $V_e$ is
a dense open subset of $X$. 
As a consequence,  the morphism
$a_e : V_e \to X$ is the inclusion
(indeed, we have $e \cdot x = e \cdot (e \cdot x)$
for any $x \in X$ such that $e \cdot x$ and
$e \cdot (e \cdot x)$ are defined.
Thus, $y = e \cdot y$ for any $y$ in a dense open
subset of $X$). Also, the rational map
$(\pr_1,a)$ is birational (see \cite[p.~515]{Demazure}).

With this notation, Weil's regularization theorem
asserts that \emph{there exists a $G$-variety
$Y$ and a birational $G$-equivariant map 
$f : X \dasharrow Y$ }, 
i.e.,  $f(g \cdot x) = g \cdot f(x)$ whenever
$(g,x) \in V$ and $f$ is defined at $g \cdot x$.
(This result was originally stated for 
a geometrically integral variety $X$, but there 
is no harm in assuming $X$ smooth).

Weil's regularization theorem is proved in
\cite[Prop.~9]{Demazure} under the
additional assumption that $a$ is 
\emph{pseudo-transitive}, 
i.e., the morphism
\[ (a,\pr_2) : V \longrightarrow X \times X,
\quad (g,x) \longmapsto (g \cdot x,x) \]
is dominant. Then $Y$ may be taken 
\emph{homogeneous} in the sense that
the analogously defined morphism
$G \times X \to X \times X$
is faithfully flat. 
We will deduce the general case
from this special one, 
which is arguably much easier. 

We start with a reduction to the case where
$k$ is \emph{algebraically closed}.
Assume that there exists a $\bar{k}$-variety
$Y$ which is $G_{\bar{k}}$-birationally
isomorphic to $X_{\bar{k}}$. 
By the principle of the finite extension
(see \cite[IV.9.1]{EGA}),
there exists a finite subextension $K/k$
of $\bar{k}/k$, a $K$-variety $Z$ equipped
with an action of $G_K$, and a
$G_K$-equivariant birational map
\[ \psi: X_K \dasharrow Z. \] 
By Theorems \ref{thm:gen} and
\ref{thm:eq}, we may further 
assume that $Z$ is projective.
Then the Weil restriction functor  
$\bfR_{K/k}(Z/K) = \bfR_K(Z)$ 
is represented by an open subscheme 
$\R_K(Z)$ of the Hilbert scheme 
parameterizing finite subschemes of
length $n$ of $Z$, where $n = [K:k]$.
In particular, the scheme $\R_K(Z)$ is 
quasi-projective (see \cite[A.5.8]{CGP}
for a direct proof of this fact).

Choose an open subset $U \subset X_K$ 
on which $\psi$ is defined, and let
$F = X_K  \setminus U$.  Also, denote by
$\pi : X_K \to X$ the projection. 
Since $\pi$ is finite and surjective, 
$\pi(F)$ is a closed subset of $X$ and 
$X \setminus \pi^{-1} \pi(F) = 
(X \setminus \pi(F))_K$ is a dense open
subset of $U$.  Replacing $X$ with
$X \setminus \pi(F)$, we may thus assume
that $\psi$ is a morphism. 
Similarly, we may further assume that 
$\psi$ is an open immersion.

Denoting by $i : V \to  G \times X$ 
the inclusion of the domain of definition
of $a$, the $G_K$-equivariance
of $\psi$ yields a commutative diagram
\[ \xymatrix{
V_K \ar[r]^-{i_K} \ar[dr]_{a_K} &
G_K \times_K X_K \ar[r]^-{\id \times \psi} &
G_K \times_K Z \ar[d]^-b \\
& X_K \ar[r]^{\psi} & Z,
} \]
where $b$ denotes the action. 
Since Weil restriction commutes 
with fibered products (see \cite[A.5.2]{CGP}),
this yields in turn a commutative diagram
\[ \xymatrix{
\R_K(V_K) \ar[r]^-{\R_K(i_K)} \ar[dr]_{\R_K(a_K)} &
\R_K(G_K) \times \R_K(X_K) \ar[r]^-{\id \times \R_K(\psi)} &
\R_K(G_K) \times \R_K(Z) \ar[d]^-{\R_K(b)} \\
& \R_K(X_K) \ar[r]^{\R_K(\psi)} & \R_K(Z),
} \]
where $\R_K(i_K)$,  $\R_K(\psi)$ and 
$\id \times \R_K(\psi)$ are open immersions
(see \cite[7.6/2]{BLR}).
We also have a commutative square 
\[ \xymatrix{
V \ar[r]^-i \ar[d]_{j_V} & 
G \times X \ar[d]^{j_G \times j_X}  \\
\R_K(V_K) \ar[r]^-{\R_K(i_K)} &
\R_K(G_K) \times \R_K(X_K),
} \]
where the adjunction morphism $j_V$ 
is a closed immersion, and likewise for
$j_G$ and $j_X$ (see \cite[A.5.7]{CGP}).
Thus, $\gamma = \R_K(\psi) \circ j_X$
is a (locally closed) immersion. 
Also, we have 
$\R_K(a_K) \circ j_V = j_X \circ a$
by functoriality (see \cite[A.5.7]{CGP} 
again).  By concatenating the two diagrams
above, this yields a commutative diagram
\[ \xymatrix{
V \ar[r]^-{i} \ar[dr]_{a} &
G \times X \ar[r]^-{\id \times \gamma} &
G \times \R_K(Z) \ar[d]^-{c} \\
& X \ar[r]^{\gamma} & \R_K(Z),
} \]
where $c$ denotes the pull-back to
$G$ of the action of $\R_K(G_K)$.
Denote by $W$ the schematic image of 
$\gamma$; then $W$ is $G$-stable, 
since the formation of the schematic image 
commutes with flat base extension. 
Thus, $\gamma$ factors through
an open immersion $\delta : X \to W$
which lies in a commutative diagram
\[ \xymatrix{
V \ar[r]^-{i} \ar[dr]_{a} &
G \times X \ar[r]^-{\id \times \delta} &
G \times W \ar[d]^-{d} \\
& X \ar[r]^{\delta} & W,
} \]
where $d$ denotes the $G$-action on $W$.
Thus, $W$ is the desired model.

So we may assume that $k$ is algebraically 
closed. At this stage we may conclude by
using the main result of \cite{Kraft}, but 
we prefer to give a fully independent proof.

We now construct the ``variety of orbits'' 
for the rational action of $G$ on $X$, by adapting
arguments from \cite[IV.1, IV.2]{Springer}.
The group $G(k)$ acts on the field of rational
functions $k(X)$ by automorphisms. Denote by
$K$ the fixed subfield; then $K$ is finitely 
generated over $k$, and hence equals $k(Y)$ 
for some variety $Y$ equipped with a dominant
rational map $f : X \dasharrow Y$. 
Replacing $X$ and $Y$ with suitable dense open 
subsets, we may assume that they are both 
affine, and $f$ is a morphism. Since the 
extension $k(X)/K$ is separable (see 
\cite[Lem.~IV.1.5]{Springer}), we may further
assume that $f$ is smooth. Finally, using generic
freeness (see \cite[IV.6.9.2]{EGA}), we may
assume that the $\cO(Y)$-module $\cO(X)$ is free.

By density of $k$-rational points, we have
$f(g \cdot x) = f(x)$ for all $g \in G$ and 
$x \in X$ such that $g \cdot x$ is defined. 
Thus, $(a,\pr_2): V \to X \times X$ yields 
a morphism 
\[ \gamma : V \longrightarrow X \times_Y X. \]
We claim that 
$\gamma$ \emph{is schematically dominant}.
As $X \times_Y X = \Spec \,
\cO(X) \otimes_{\cO(Y)} \cO(X)$,
this is equivalent to the injectivity of
$\gamma^{\#} : \cO(X) \otimes_{\cO(Y)} \cO(X)
\to \cO(V)$.
To show this, let 
$u_1,\ldots,u_r,v_1, \ldots, v_r \in \cO(X)$
such that 
$\gamma^{\#}(\sum_{i=1}^r u_i \otimes v_i) = 0$. 
We may assume that $u_1,\ldots,u_r$ are part 
of a basis of the $\cO(Y)$-module $\cO(X)$;
then they are linearly independent over $K$.
We then have 
$\sum_{i=1}^r u_i(g \cdot x)  v_i(x) = 0$
for all $x \in X$ and $g \in G(k)$ such that
$g \cdot x$ is defined. Thus, we have 
$\sum_{i=1}^r (g \cdot u_i)  v_i = 0$
in $k(X)$ for any $g \in G(k)$. By
\cite[Lem.~IV.1.5]{Springer} again, it follows
that $v_1 = \cdots = v_r = 0$. This proves
the claim.

Denote by $\eta$ the generic point of $Y$;
then the generic fiber $X_{\eta}$ is a smooth
$K$-variety equipped with a rational action of 
the $K$-group $G_K$. As the morphism
$\eta \to Y$ is flat and $\gamma$ is
schematically dominant, the base change
$\gamma_{\eta} : V \times_Y \eta \to 
(X \times_Y X) \times_Y \eta$ 
is schematically dominant as well
(see \cite[IV.11.10.5]{EGA}). So
the rational map
\[ (a,\pr_2) : G_K \times_K X_{\eta}
\dasharrow X_{\eta} \times_K X_{\eta} \]
is dominant, i.e., the rational
action of $G_K$ on $X_{\eta}$ is
almost transitive. By \cite[Prop.~9]{Demazure}, 
there exists a $G_K$-homogeneous $K$-variety 
$Z$ which is $G_K$-equivariantly birational 
to $X_{\eta}$. In particular, $K(Z) = k(X)$.

It remains to ``regularize'' $Z$.
We first treat the case where $G$ is 
\emph{affine}, or equivalently linear. 
Then $Z$ is $G_K$-quasi-projective;
equivalently, 
there exists a homomorphism
$\rho: G_K \to \GL_{n,K}$ and
an immersion $Z \hookrightarrow \bP^{n-1}_K$
which is equivariant relative to $\rho$
(see e.g. \cite[Cor.~2.14]{Br17a}).
Denote by $(a_{ij})_{1 \leq i,j \leq n}$
the matrix coefficients of $\rho$ and
by $d$ their determinant. Then
$a_{ij}, d^{-1} \in \cO(G_K) = \cO(G) \otimes_k K$.
We may choose a finitely generated subalgebra
$R$ of $K$ which contains the $a_{ij}$ and
$d^{-1}$, and has fraction field $K$. Let 
$Y = \Spec(R)$, then $\rho$ extends to
a unique homomorphism of $Y$-group schemes
\[ \rho_Y : G_Y = G \times Y 
\longrightarrow \GL_{n,Y}. \] 
Denote by $W$ the closure of $Z$ in 
$\bP^n_Y = \bP^n \times Y$; then $W$ is stable 
under the action of $G_Y$ via $\rho_Y$. 
The resulting morphism 
$G_Y \times_Y W \to W$ yields a $G$-action
on $W$, and in turn the desired regularization.

For an arbitrary (smooth) algebraic group $G$,
we first reduce to the case where $G$ 
is \emph{connected} by adapting the proof of 
\cite[Thm.~1]{Rosenlicht}. Denote by $G^0$
the neutral component of $G$, and assume that
$X$ is $G^0$-birationally isomorphic to
a $G^0$-variety. We may then assume that
$X$ is a $G^0$-variety, i.e., $g \cdot x$
is defined for any $g \in G^0$ and $x \in X$. 
Also, we may choose $g_1,\ldots,g_r \in G(k)$ 
such that $G = \bigcup_{i=1}^r G^0 \, g_i$
(since $k$ is algebraically closed). Then 
$g_i \cdot x$ is defined for any $i = 1,\ldots, r$ 
and for any 
$x \in \bigcap_{i=1}^r  V_{g_i} = X \setminus Y$, 
where $Y \subsetneq X$ is closed. Thus,
$h \cdot (g_i \cdot x)$ is defined for any
$h \in G^0$ and $i,x$ as above, i.e., 
$g \cdot x$ is defined for all $g \in G$ 
and $x \in X \setminus Y$.

Let $Z = \bigcap_{h \in G^0(k)} h \cdot Y$;
then $Z$ is a closed subset of $X$, 
stable by $G^0(k)$.
Since $G$ is smooth, $G^0(k)$ is schematically dense 
in $G^0$ and hence $Z$ is $G^0$-stable. We claim that 
$g \cdot x$ is defined for any $g \in G$ and 
$x \in X \setminus Z$. Indeed, 
$x \in X \setminus h \cdot Y$ for some $h \in G^0(k)$.
Thus, $h^{-1} \cdot x$ is defined and lies in 
$X \setminus Y$. So $gh \cdot (h^{-1} \cdot x)$ is
defined. 
By using (iii), this yields the claim.

Next, let $W$ be the subset of $X$ consisting of
those $x$ such that $g_i \cdot x$ is defined and
lies in $Z$ for some $i$. Then $W$ is closed in 
$X$: indeed, $W$ contains $Z$ and 
$W \setminus Z = 
\bigcap_{i=1}^r g_i^{-1}(X \setminus Z)$
is closed in $X \setminus Z$. 
Also, $W$ is stable by $G^0(k)$
(since the latter is normalized by the $g_i$)
and hence by $G^0$.
For any $g \in G$ and $x \in X \setminus W$,
we have that $g \cdot x$ is defined (as $x \notin Z$)
and $g \cdot x \notin W$ (otherwise, 
$g_i g \cdot x \in Z$ for some $i$, and hence
$g_j \cdot x \in Z$ for $j$ such that 
$g_i g \in G^0 g_j$. So $x \in W$, a contradiction).
In other terms, $g$ is defined at any point of 
$X \setminus W$, and stabilizes this open subset.
As a consequence, each $g_i$ yields an automorphism
of $X \setminus W$. So the $G^0$-action on 
$X \setminus W$ extends to a $G$-action.

Thus, we may assume that ($k$ is algebraically 
closed and) $G$ is connected.
In view of \cite[Prop.~8]{Demazure}, we may further
assume that the rational action of $G$ on $X$ is 
\emph{faithful}, i.e., its schematic kernel is 
trivial. We have a unique exact sequence
\[ 1 \longrightarrow N \longrightarrow G
\stackrel{\pi}{\longrightarrow} A 
\longrightarrow 1, \]
where $N$ is smooth, connected and affine, 
and $A$ is an abelian variety (see e.g.
\cite[Thm.~8.28]{Milne}). Also, there exists
a finite Galois extension of fields
$L/K$ such that $Z_L$ has a $L$-rational
point. Then $Z_L$ is a homogeneous
space $G_L/H$, where $H$ is affine
(see \cite[Prop.~8.9]{Milne}). Consider the
exact sequence 
\[ 1 \longrightarrow N_L \longrightarrow G_L
\longrightarrow A_L \longrightarrow 1. \]
The schematic image of $H$ in $A_L$ 
is an affine subgroup, and hence is finite.
Thus, the homogeneous space 
$G_L/N_L H$ is a quotient of $A_L$ by 
a finite subgroup. Also, the natural
morphism $G_L/H \to G_L/N_L H$ is the
geometric quotient by the action of 
$N_L$ via left multiplication on $G_L$.
By Galois descent, it follows that
the geometric quotient $Z \to Z/N_K$ 
exists, and $Z/N_K$ is a torsor under
the quotient of $A_K$ by a finite subgroup. 
Since $Z$ is smooth,
Lemma \ref{lem:tabelian} yields
an $A_K$-equivariant morphism 
$Z/N_K \to A_K/A[n]_K$ for 
some positive integer $n$.
Denote by $I$ the pull-back of $A[n]$
under the homomorphism $\pi: G \to A$; 
then $N \subset I \subset G$ and 
$I/N$ is finite. Moreover, we have
a $G_K$-equivariant morphism 
$Z \to Z/N_K \to A_K/A[n]_K = G_K/I_K$
and hence $Z \simeq G_K \times^{I_K} W$
for some $I_K$-scheme $W$. Since $Z$ 
is normal and quasi-projective, it is
$I_K$-quasi-projective (see again 
\cite[Cor.~2.14]{Br17a}); thus, so is $W$.
By arguing as in the affine case, one
obtains an affine variety $Y$ with function
field $K$, and a scheme $V$ equipped with 
an action of $I$ and an $I$-invariant morphism 
$V \to Y$ having generic fiber $W$; moreover,
$V$ is $I$-quasi-projective. Thus, 
$G \times^I V$ exists and yields the desired
regularization.

\section{Actions of non-smooth algebraic groups}
\label{sec:ns}

Throughout this section, the ground field $k$  
has characteristic $p > 0$
(otherwise every algebraic group is smooth). 
We first construct actions having no normal
projective model:

\begin{example}\label{ex:nonsmooth}
We consider torsors under the ``smallest'' 
non-smooth algebraic group $G = \alpha_p$,
the kernel of the Frobenius homomorphism
$F : \bG_a \to \bG_a$,  $z \mapsto z^p$.

Let $C$ be an integral affine curve and 
consider the product $C \times \bA^1$ 
with projections $x$, $y$. Choose 
$f \in \cO(C)$ 
and denote by $X \subset C \times \bA^1$ 
the zero subscheme of $y^p - f(x)$. 
So we have a cartesian square
\[ \xymatrix{
X \ar[r]^{y} \ar[d]_x & \bA^1 \ar[d]^F \\
C \ar[r]^-{f} & \bA^1. 
}\]

The group $G$ acts on $C \times \bA^1$ via
$g \cdot (x, y)  = (x, y + g)$
and this action stabilizes $X$.
The projection $x : X \to C$ is a $G$-torsor,
the pull-back of $F$ under $f$.
By \cite[III.4.6.7]{DG}, every $G$-torsor 
over an affine scheme is obtained via 
this construction. 

The affine curve $X$ is integral if and 
only if $f$ is not a $p$th power in 
the function field $k(C)$. When $C$ is
normal, this is equivalent to $f$ 
not being a $p$th power in $\cO(C)$, 
and hence to the non-triviality of the 
torsor $x$ (see \cite[III.4.6.7]{DG} again).

We now assume for simplicity that
$k$ is algebraically closed and $C$ is
smooth. Then $X$ is non-normal if and only if
the morphism $f : C \to \bA^1$ is ramified; 
this holds e.g. if $f$ has a multiple factor when
expressed in local coordinates at some point.

Assuming in addition that $X$ is 
integral and non-normal, we claim that 
\emph{the $G$-action on $X$ does not lift 
to an action on the normalization $\tilde{X}$}. 
Otherwise, $G$ acts freely on $\tilde{X}$
and hence there exists a $G$-torsor
$q : \tilde{X} \to \tilde{X}/G$, where
$\tilde{X}/G$ is an integral affine curve
(see \cite[III.2.2.3, III.2.6.1]{DG}). 
Thus, the normalization morphism
$\nu: \tilde{X} \to X$ descends to
a morphism $\mu : \tilde{X}/G \to C$;
equivalently, we have a commutative square
\[ \xymatrix{
\tilde{X} \ar[r]^{\nu} \ar[d]_q & X \ar[d]^x \\
\tilde{X}/G \ar[r]^-{\mu} & C.
}\] 
As the vertical arrows are $G$-torsors, 
this square is cartesian. Since $\nu$ 
is finite, so is $\mu$ by descent. Also, 
since $\nu$ is birational and every open subset 
of $X$ is $G$-stable, it follows that
$\mu$ is birational as well. As $C$ is smooth, 
$\mu$ is an isomorphism. So $\nu$ is 
an isomorphism as well, a contradiction. 
This proves our claim.

Next, we show that 
\emph{$X$ admits no normal projective model}. 
Indeed, $X$ has a projective model $\bar{X}$, 
its closure in $\bar{C} \times \bP^1$ where 
$\bar{C}$ denotes the smooth projective 
completion of $C$, and $G$ acts on 
$\bar{C} \times \bP^1$ 
via
$g \cdot (x,[y:z]) = (x, [y + g z : z])$.
If $X$ has a normal projective model 
$X'$, then the $G$-equivariant rational map
of curves $X' \dasharrow X$ 
yields an equivariant morphism 
$\varphi: X' \to \bar{X}$. Thus, 
$\varphi$ is the normalization, and hence 
so is its pull-back $\varphi^{-1}(X) \to X$, 
a contradiction to the claim.
\end{example}

Next, we discuss the notion of rational action
in the setting of non-smooth groups,
for which we could locate no convenient reference.

As in \cite[9.6]{GW}, we define a 
\emph{rational map} $f : X \dasharrow Y$
(where $X$ and $Y$ are schemes) as an
equivalence class of pairs $(U,g)$,
where $U$ is a schematically dense open subset
of $X$, and $g : U \to Y$ is a morphism. Here 
$(U,g)$ is said to be equivalent to $(V,h)$ 
if there exists a schematically dense open
subset $W$ of $U \cap V$ such that 
$g \vert_W = h \vert_W$. (Equivalently,
$f$ is a pseudo-morphism in the sense of 
\cite[IV.20.2]{EGA}). There exists a largest
such subset $U$, the \emph{domain of definition} 
of $f$.

Next, a rational action of an algebraic group
$G$ on a smooth variety $X$ is defined as 
in Section \ref{sec:wrt},  by replacing ``dominant'' 
with ``schematically dominant'' in (i).
The arguments in the beginning of 
\cite[\S 3]{Demazure} adapt readily with
this replacement, and yield the equivalence 
of (ii) and (iii). Also, for any $g \in G$
with residue field $\kappa(g)$, we may still
identify $V \cap \pr_1^{-1}(g)$ with a
schematically dominant open subset 
$V_g$ of $X_{\kappa(g)}$, and the morphism
$V_g \to X$, $g \mapsto g \cdot x$ 
is (schematically) dominant. Here $V$ denotes 
again the domain of definition of the action.
Thus, we still have that $e \cdot x = x$
for all $x \in V_e$.

We may now state:

\begin{theorem}\label{thm:weil}
Let $X$ be a smooth variety equipped with 
a rational action of an algebraic group 
$G$. Then $X$ is $G$-birationally isomorphic 
to a projective $G$-variety.
\end{theorem}

\begin{proof}
We first consider the case where $G$ acts on $X$.
We may choose a positive integer $n$ such that 
$G/G_n$ is smooth, where $G_n$ denotes 
the $n$th Frobenius kernel
(see e.g. \cite[Prop.~2.9.2]{Br17b}).
By \cite[Lem.~2.8]{Br17a}, 
the quotient $X/G_n$ exists and is a $G/G_n$-variety.
Using Theorem \ref{thm:gen}, we may further assume 
that $X/G_n$ is regular and quasi-projective.
We now claim that $X$ admits a $G$-equivariant
projective completion.

By \cite[Lem.~2.9]{Br17a}, $X/G_n$ admits an
ample $H$-linearized line bundle, 
where $H$ denotes the smallest normal subgroup
of $G$ such that $G/H$ is proper., 
and $H$ acts on $X/G_n$ via its quotient
$H/H \cap G_n = H/H_n$.
Since the quotient morphism $X \to X/G_n$ is finite, 
$X$ also admits an ample $H$-linearized line bundle.
The claim follows from this 
by arguing as in the proof of Theorem \ref{thm:eq}
(except for the final step of normalization).

Next, we turn to the general case, where
$G$ acts rationally on $X$. By the above claim,
it suffices to show that $X$ is $G$-birationally
isomorphic to a $G$-variety. For this, we
may assume that $k$ is algebraically closed
by the first reduction step in the proof
of Weil's regularization theorem, which extends 
without any change. 
We then have $G = G_n \, G_{\red}$,
where $G_{\red}$ denotes the largest smooth
subgroup of $G$ (the closure of $G(k)$),
see e.g. \cite[Lem.~2.8.6]{Br17b}. 
Using the regularization theorem, we may further 
assume that $X$ is a $G_{\red}$-variety.
Denote by $V_n \subset G_n \times X$
the domain of definition of the rational 
action $a_n : G_n \times X \dasharrow X$. Then
$V_n$ is schematically dense in $G_n \times X$,
and stable by $G(k)$ (acting on $G_n \times X$
via $g \cdot (h,x) = (ghg^{-1}, g \cdot x)$)
in view of the condition 
(iii) in Section \ref{sec:wrt}. Thus,
$V_n \cap \pr_1^{-1}(e)$ is identified
with a schematically dense open subset 
$(V_n)_e$ of $X$, which is stable by $G(k)$ as well. 
As a consequence, $(V_n)_e$ is stable by $G_{\red}$. 
We may thus assume that $(V_n)_e = X$. 

Since $G_n$ is infinitesimal, the projection 
$\pr_2: G_n \times X \to X$ induces a homeomorphism
of the underlying topological spaces,  
with inverse homeomorphism $(e,\id_X)$. 
As $V_n$ is open  in $G_n \times X$, it follows 
that $V_n = G_n \times X$ as schemes, i.e., 
$G_n$ acts on $X$. Since $G_{\red}$ also
acts on $X$, the domain of definition
of the rational $G$-action is the whole $X$
by (iii) again. So $X$ is a $G$-variety.
\end{proof}

\medskip

\noindent
{\bf Acknowledgements.} Many thanks to Igor Dolgachev, 
Antoine V\'ezier and Susanna Zimmermann 
for stimulating discussions and email exchanges  
on the topics of this paper. I also thank 
J\'er\'emy Blanc, Pascal Fong and Ronan Terpereau
for very helpful comments on a first version.

\bibliographystyle{amsalpha}

\medskip

\noindent
Universit\'e Grenoble Alpes, 
Institut Fourier, 
100 rue des Math\'ematiques,
38610 Gi\`eres, France

\end{document}